\documentclass[a4j,11pt]{article}

\usepackage{amsmath}
\usepackage{amssymb}
\usepackage{amsthm}
\usepackage{array}
\usepackage{amscd}
\usepackage{cases}
\usepackage{bm}
\usepackage[all]{xy}
\usepackage{authblk}
\usepackage{graphicx}
\usepackage{lscape}

\theoremstyle{plain}
\newtheorem{thm}{Theorem}[section]
\newtheorem{prop}[thm]{Proposition}
\newtheorem{lem}[thm]{Lemma}

\newtheorem{conj}[thm]{Conjecture}
\newtheorem*{conj*}{Conjecture}
\theoremstyle{definition}
\newtheorem{df}{Definition}

\theoremstyle{definition}
\newtheorem{rem}{Remark}
\newtheorem*{rem*}{Remark}
\newtheorem{eg}[thm]{Example}

\theoremstyle{remark}

\numberwithin{equation}{section}

\title{Four-dimensional Painlev\'e-type difference equations}
\author{Hiroshi Kawakami\thanks{\texttt{kawakami@gem.aoyama.ac.jp}}}
\affil{College of Science and Engineering, Aoyama gakuin university,
5-10-1 Fuchinobe, Chuo-ku, Sagamihara-shi, Kanagawa 252-5258, Japan.}
\date{}

\begin{document}
\maketitle

\begin{abstract}
We focus on Fuchsian equations with four accessory parameters and three singular points.
We see that the Fuchsian equations admit a ``degeneration scheme'' in some sense,
which is expected to give rise to a degeneration scheme of discrete isomodromic deformation equations with four-dimensional phase space.
We compute an example of discrete isomonodromic deformation equations of a certain Fuchsian equation.
\end{abstract}

\paragraph{Mathematics Subject Classifications (2010).}34M55, 34M56, 33E17
\paragraph{Key words.}isomonodromic deformation, Painlev\'e-type equation, integrable system, degeneration.

%\tableofcontents

\section{Introduction}\label{sec:intro}
The Painlev\'e equations are second order non-linear ordinary differential equations
discovered by Painlev\'e and Gambier in an attempt to find new special functions.
Recently, there have been many studies concerning generalizations of the Painlev\'e equations
based on various aspects of them, such as
the affine Weyl group symmetry, the spaces of initial conditions, and so on.
Then it is natural to ask: How can we describe those generalizations in a unified way?
The important fact is that 
they can also be obtained as compatibility conditions of linear differential equations.
In other words, they can be regarded as \textit{isomonodromic deformation equations} of some linear differential equations.
Thus %, for a unified understanding of those generalizations, 
it is important to give a good description of the whole set of isomonodromic deformation equations.

An isomonodromic deformation is a deformation of a linear differential equation
which keeps its ``monodromy data'' unchanged.
Isomonodromic deformations fall into two classes: the continuous deformation and the discrete deformation.
When we consider the continuous isomonodromic deformation of a linear equation,
we can choose positions of its singular points and
coefficients of its HTL canonical forms (see Definition~\ref{def:HTL}) except the residue matrices %at irregular singular points
as deformation parameters.
Then we obtain a system of non-linear differential equations satisfied by the coefficients of the linear equation.
In the course of a continuous isomonodromic deformation,
the residue matrices of the HTL forms %parameters of the linear equation 
(so-called ``exponents of formal monodromy'')
%(or characteristic exponents at regular singular points)
stay constant.
However, we can consider a discrete change of the exponents, 
which does not change its monodromy data.
Such a discrete deformation of a linear equation is called a Schlesinger transformation,
which is expressed as a system of non-linear difference equations.
See \cite{JMU, JM} for details on isomonodromic deformations.
In what follows 
we use the term \textit{Painlev\'e-type} equations synonymously with isomonodromic deformation equations.

It is well-known that Painlev\'e-type differential equations can be written in Hamiltonian form.
In particular, the dimensions of the phase spaces of Painlev\'e-type differential equations are even numbers.
Recently, as a first step toward a comprehensive understanding of isomonodromic deformation equations,
a classification of the Painlev\'e-type differential equations with four-dimensional phase space was obtained~\cite{K1, K2, K3, KNS}.
The study of four-dimensional Painlev\'e-type equations
owes much to the classification of Fuchsian equations with four accessory parameters by Oshima~\cite{Os}.

\begin{thm}[Oshima]\label{thm:oshima}
Any irreducible Fuchsian equation with four accessory parameters can be transformed into one of the following 13 equations
by a finite iteration of additions and middle convolutions.
\[
\begin{tabular}{r|ccc}
\#sing. \!\!pt. $= 5$ & $11,11,11,11,11$ & & \\
\hline
$4$ & $21,21,111,111$ & $31,22,22,1111$ & $22,22,22,211$ \\
\hline
$3$ & $211,1111,1111$ & $221,221,11111$ & $32,11111,11111$ \\
      & $222,222,2211$   & $33,2211,111111$ & $44,2222,22211$ \\
      & $44,332,11111111$ & $55,3331,22222$ & $66,444,2222211$
\end{tabular}
\]
\end{thm}
\noindent
The tuples of integers in the above table are called spectral types of Fuchsian equations,
which represent multiplicities of characteristic exponents (see Section~\ref{sec:RS}).
We note that the number of accessory parameters of a linear equation coincide with the dimension of the phase space of
the corresponding Painlev\'e-type equation.
Fuchsian equations have only the position of singular points as continuous deformation parameters.
If a Fuchsian equation has $N$ singular points on the Riemann sphere,
then three of the $N$ points can be mapped to $0,1,\infty$ via the M\"obius transformation.
Thus the number of essential deformation parameters of the Fuchsian equation is $N-3$.
The Painlev\'e-type differential equation corresponding to the Fuchsian equation of spectral type $11,11,11,11,11$
is the Garnier system in two variables. %, which is classically known.
The Painlev\'e-type equations corresponding to the Fuchsian equations with four singular points in the above table
were clarified by Sakai~\cite{Sak2}:
%the Garnier system ($11,11,11,11,11$), 
the Fuji-Suzuki system~\cite{FS1, Ts} (corresponding to $21,21,111,111$), 
the Sasano system~\cite{FS2, Ss} (corresponding to $31,22,22,1111$),
and the sixth matrix Painlev\'e system~\cite{B2, Sak2} (corresponding to $22,22,22,211$).
In \cite{K1, K2, K3, KNS}, %the notion of spectral type for non-Fuchsian equations was introduced,
the degeneration scheme of the four-dimensional Painlev\'e-type differential equations
was obtained starting from the four Painlev\'e-type equations. 
This provides a classification of four-dimensional Painlev\'e-type differential equations.

\vspace{5mm}

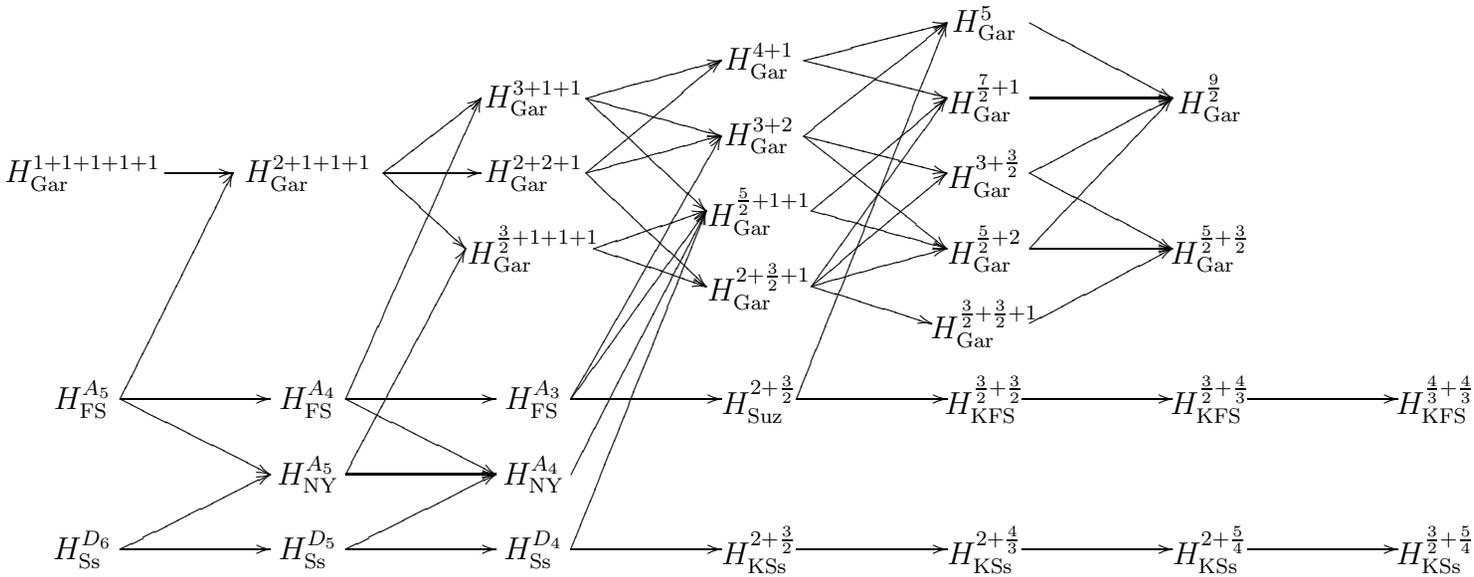
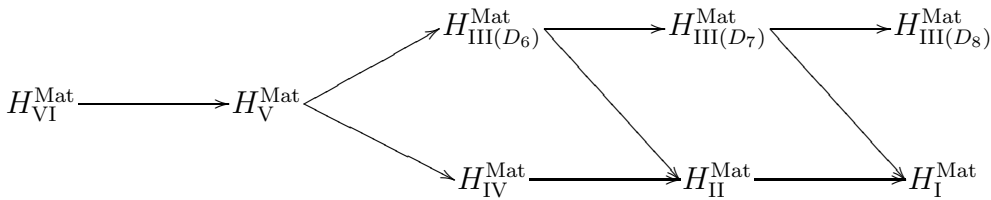
\rotatebox{90}{\begin{minipage}{\textheight}
\centering
%\begin{landscape}
{\large
\begingroup
\renewcommand{\arraystretch}{1.5}
\begin{xy}
{(0,-12) *{H_{\mathrm{Gar}}^{1+1+1+1+1}}},
{\ar (11,-12);(20,-12)},
{(30,-12) *{H_{\mathrm{Gar}}^{2+1+1+1}}},
{\ar (40,-12);(53,-2)},
{\ar (40,-12);(53,-12)},
{\ar (40,-12);(51,-22)},
{(60,-2) *{H_{\mathrm{Gar}}^{3+1+1}}},
{(60,-12) *{H_{\mathrm{Gar}}^{2+2+1}}},
{(60,-22) *{H_{\mathrm{Gar}}^{\frac32+1+1+1}}},
{\ar (67,-2);(85,3)},
{\ar (67,-2);(85,-7)},
{\ar (67,-2);(83,-17)},
{\ar (67,-12);(85,3)},
{\ar (67,-12);(85,-7)},
{\ar (67,-12);(83,-27)},
{\ar (68,-22);(83,-17)},
{\ar (68,-22);(83,-27)},
{(90,3) *{H_{\mathrm{Gar}}^{4+1}}},
{(90,-7) *{H_{\mathrm{Gar}}^{3+2}}},
{(90,-17) *{H_{\mathrm{Gar}}^{\frac52+1+1}}},
{(90,-27) *{H_{\mathrm{Gar}}^{2+\frac32+1}}},
{\ar (96,3);(115,8)},
{\ar (96,3);(115,-2)},
{\ar (96,-7);(115,8)},
{\ar (96,-7);(115,-12)},
{\ar (96,-7);(115,-22)},
{\ar (97,-17);(115,-2)},
{\ar (97,-17);(115,-22)},
{\ar (97,-27);(115,-2)},
{\ar (97,-27);(115,-12)},
{\ar (97,-27);(115,-22)},
{\ar (97,-27);(113,-32)},
{(120,8) *{H_{\mathrm{Gar}}^5}},
{(120,-2) *{H_{\mathrm{Gar}}^{\frac72+1}}},
{(120,-12) *{H_{\mathrm{Gar}}^{3+\frac32}}},
{(120,-22) *{H_{\mathrm{Gar}}^{\frac52+2}}},
{(120,-32) *{H_{\mathrm{Gar}}^{\frac32+\frac32+1}}},
{\ar (126,8);(145,-2)},
{\ar (126,-2);(145,-2)},
{\ar (126,-12);(145,-2)},
{\ar (126,-12);(145,-22)},
{\ar (126,-22);(145,-2)},
{\ar (126,-22);(145,-22)},
{\ar (126,-32);(145,-22)},
{(150,-2) *{H_{\mathrm{Gar}}^{\frac92}}},
{(150,-22) *{H_{\mathrm{Gar}}^{\frac52+\frac32}}},
{(0,-42) *{H_{\mathrm{FS}}^{A_5}}},
{\ar (5,-42);(25,-42)},
{\ar (5,-42);(25,-52)},
{\ar (5,-42);(20,-12)},
{(30,-42) *{H_{\mathrm{FS}}^{A_4}}},
{\ar (35,-42);(55,-42)},
{\ar (35,-42);(55,-52)},
{\ar (35,-42);(53,-2)},
{(60,-42) *{H_{\mathrm{FS}}^{A_3}}},
{\ar (65,-42);(85,-42)},
{\ar (65,-42);(83,-17)},
{\ar (65,-42);(85,-7)},
{(90,-42) *{H_{\mathrm{Suz}}^{2+\frac32}}},
{\ar (95,-42);(115,-42)},
{\ar (95,-42);(115,8)},
{(120,-42) *{H_{\mathrm{KFS}}^{\frac32+\frac32}}},
{\ar (125,-42);(145,-42)},
{(150,-42) *{H_{\mathrm{KFS}}^{\frac32+\frac43}}},
{\ar (155,-42);(175,-42)},
{(180,-42) *{H_{\mathrm{KFS}}^{\frac43+\frac43}}},
{(30,-52) *{H_{\mathrm{NY}}^{A_5}}},
{\ar (35,-52);(55,-52)},
{\ar (35,-52);(51,-22)},
{(60,-52) *{H_{\mathrm{NY}}^{A_4}}},
{\ar (65,-52);(83,-17)},
{(0,-62) *{H_{\mathrm{Ss}}^{D_6}}},
{\ar (5,-62);(25,-62)},
{\ar (5,-62);(25,-52)},
{(30,-62) *{H_{\mathrm{Ss}}^{D_5}}},
{\ar (35,-62);(55,-62)},
{\ar (35,-62);(55,-52)},
{(60,-62) *{H_{\mathrm{Ss}}^{D_4}}},
{\ar (65,-62);(85,-62)},
{\ar (65,-62);(83,-17)},
{(90,-62) *{H_{\mathrm{KSs}}^{2+\frac32}}},
{\ar (95,-62);(115,-62)},
{(120,-62) *{H_{\mathrm{KSs}}^{2+\frac43}}},
{\ar (125,-62);(145,-62)},
{(150,-62) *{H_{\mathrm{KSs}}^{2+\frac54}}},
{\ar (155,-62);(175,-62)},
{(180,-62) *{H_{\mathrm{KSs}}^{\frac32+\frac54}}},
{(30,-100) *{H^{\mathrm{Mat}}_{\mathrm{VI}}}},
{\ar (35,-100);(55,-100)},
{(60,-100) *{H^{\mathrm{Mat}}_{\mathrm{V}}}},
{\ar (65,-100);(83,-90)},
{\ar (65,-100);(85,-110)},
{(90,-90) *{H^{\mathrm{Mat}}_{\mathrm{III}(D_6)}}},
{\ar (97,-90);(113,-90)},
{\ar (97,-90);(115,-110)},
{(90,-110) *{H^{\mathrm{Mat}}_{\mathrm{IV}}}},
{\ar (95,-110);(115,-110)},
{(120,-90) *{H^{\mathrm{Mat}}_{\mathrm{III}(D_7)}}},
{\ar (127,-90);(143,-90)},
{\ar (127,-90);(145,-110)},
{(120,-110) *{H^{\mathrm{Mat}}_{\mathrm{II}}}},
{\ar (125,-110);(145,-110)},
{(150,-90) *{H^{\mathrm{Mat}}_{\mathrm{III}(D_8)}}},
{(150,-110) *{H^{\mathrm{Mat}}_{\mathrm{I}}}}
\end{xy}
\endgroup
}
%\end{landscape}
\end{minipage}}

\vspace{5mm}

\noindent
%\begin{rem}
Symbols such as $H^{\mathrm{Mat}}_{\mathrm{VI}}$ stand for Hamiltonians for four-dimensional Painlev\'e-type equations.

On the other hand, Fuchsian equations with three singular points admit only trivial continuous isomonodromic deformations,
but admit non-trivial discrete isomonodromic deformations.
In the case of two accessory parameters, %(i.e., original Painlev\'e equations),
discrete isomonodromic deformations of Fuchsian equations with three singular points
yield the additive difference Painlev\'e equations
that do not correspond to B\"acklund transformations of differential Painlev\'e equations~\cite{B, DST, DT}.
In this paper, we focus on the Fuchsian equations with three singular points
and four accessory parameters.
By considering the discrete isomonodromic deformation of these Fuchsian equations,
we can obtain four-dimensional Painlev\'e-type difference equations.

The organization of this paper is as follows.
In Section~\ref{sec:LDE}, 
we briefly review some notions related to linear differential equations.
In Section~\ref{sec:deg_scheme}, 
we show that the Fuchsian equations with three singular points and four accessory parameters
admit ``degeneration scheme'' in some sense,
and thereby the above degeneration scheme can be extended further upstream.
In Section~\ref{sec:d-Garnier},
we compute Schlesinger transformations of the linear equation of spectral type $211,1111,1111$ as an example.
The system of difference equations thus obtained can be regarded as a discrete analogue of the Garnier system in two variables.

\section{Preliminaries}\label{sec:LDE}

\subsection{HTL canonical forms}
Let $A(x)$ be a matrix-valued function in $x$.
The transformation
\[
A(x) \mapsto P[A(x)]:=P A(x) P^{-1}+\frac{dP}{dx}P^{-1}
\]
by an invertible matrix $P=P(x)$ is called the \textit{gauge transformation}.
This corresponds to the change of the dependent variable $Z=PY$ of a system of linear differential equations
\begin{align*}
\frac{dY}{dx}=A(x)Y.
\end{align*}

We consider a system of linear differential equations with rational function coefficients
\begin{equation}\label{eq:rational_LDE}
\frac{dY}{dx}=
\left(\sum_{\nu=1}^n\sum_{k=0}^{r_{\nu}}\frac{A_{\nu}^{(k)}}{(x-u_{\nu})^{k+1}}
+\sum_{k=1}^{r_{\infty}}A_{\infty}^{(k)}x^{k-1}
\right)Y,\quad
A_*^{(k)} \in M(m, \mathbb{C}).
\end{equation}
Note that the residue matrix $A_\infty^{(0)}$ at $x=\infty$ is given by
\[
A_\infty^{(0)}=-\sum_{\nu=1}^n A_\nu^{(0)}.
\]
The system can be transformed into the ``canonical form'' at each singular point.

The system (\ref{eq:rational_LDE}) has singularity at $x=u_\nu \, (\nu=1, \ldots, n)$ and $x=\infty=:u_0$.
Set $z=x-u_{\nu} \, (\nu=1, \ldots, n)$ or $z=1/x$.
We consider the system around $z=0$:
\begin{equation*}
\frac{dY}{dz}=\left(
\frac{A_0}{z^{r+1}}+\frac{A_1}{z^{r}}+\cdots+A_{r+1}+A_{r+2} z+\cdots
\right)Y.
\end{equation*}
Let $\mathcal{P}_z=\cup_{p > 0}\mathbb{C}(\!(z^{1/p})\!)$ be the field of Puiseux series
where $\mathbb{C}(\!(t)\!)$ is the field of formal Laurent series in $t$.

\begin{df}[HTL canonical form]\label{def:HTL}
An element
\begin{equation*}
\frac{T_0}{z^{l_0}}+\frac{T_1}{z^{l_1}}+\cdots+\frac{T_{s-1}}{z^{l_{s-1}}}+\frac{\Theta}{z^{l_s}}
\end{equation*}
in $M(m, \mathcal{P}_z)$ satisfying the following conditions:
\begin{itemize}
\item $l_j$ is a rational number, $l_0 > l_1 > \cdots > l_{s-1} > l_s=1$,
\item $T_0, \ldots, T_{s-1}$ are commuting diagonalizable matrices,
\item $\Theta$ is a (not necessarily diagonalizable) matrix that commutes with all $T_j$'s
\end{itemize}
is called an \textit{HTL canonical form}, or \textit{HTL form} for short.
\qed
\end{df}

\begin{thm}[Hukuhara~\cite{Huk}, Turrittin~\cite{Tur}, Levelt~\cite{Lev}]
For any
\begin{equation*}
A(z)=\frac{A_0}{z^{r+1}}+\frac{A_1}{z^{r}}+\cdots \in M(m, \mathbb{C}(\!(z)\!)),
\end{equation*}
there exists $P \in \mathrm{GL}(m, \mathcal{P}_z)$ such that $P[A(z)]$ is an HTL form
\[
\frac{T_0}{z^{l_0}}+\frac{T_1}{z^{l_1}}+\cdots+\frac{T_{s-1}}{z^{l_{s-1}}}
+\frac{\Theta}{z}.
\]
Here $l_0,\ldots,l_{s-1}$ are uniquely determined only by $A(z)$.

If
\begin{equation*}
\frac{\tilde{T}_0}{z^{l_0}}+\frac{\tilde{T}_1}{z^{l_1}}+\cdots+\frac{\tilde{T}_{s-1}}{z^{l_{s-1}}}
+\frac{\tilde{\Theta}}{z}
\end{equation*}
is another HTL form of the same $A(z)$,
then there exist $g \in \mathrm{GL}(m, \mathbb{C})$ and $k \in \mathbb{Z}_{\ge 1}$ such that
\begin{equation*}
\tilde{T}_j=gT_j g^{-1}, \quad \exp(2\pi i k \tilde{\Theta})=g\exp(2\pi i k \Theta)g^{-1}
\end{equation*}
hold.
\end{thm}
\noindent
The number $l_0-1$ is called the \textit{Poincar\'e rank} of the singular point.
If there is a rational number $l_j \in \mathbb{Q} \setminus \mathbb{Z}$, the singular point is called a \textit{ramified} irregular singular point.
A linear system is said to be of \textit{ramified type} if the system has a ramified irregular singular point.
We consider linear systems of unramified type below.

\subsection{Riemann schemes and spectral types}\label{sec:RS}
Consider an HTL form
\begin{equation*}
\frac{T_0}{z^{b+1}}+\frac{T_1}{z^b}+\cdots+\frac{T_{b-1}}{z^2}+\frac{\Theta}{z} \quad (b \in \mathbb{Z}_{\ge 0}).
\end{equation*}
Here we assume $T_j$'s and $\Theta$ are in Jordan canonical form.
When $\Theta$ is a diagonal matrix, we denote this HTL form by
\[
\begin{array}{c}
x=u_i \\
\overbrace{\begin{array}{ccccc}
   t^0_1  & t^1_1 & \ldots & t^{b-1}_1  & \theta_1\\
   \vdots & \vdots    &        & \vdots & \vdots \\
   t^0_m  & t^1_m & \ldots & t^{b-1}_m  & \theta_m
           \end{array}}
\end{array}
\]
where $T_j=\mathrm{diag}(t^j_1,\ldots,t^j_m)$, $\Theta=\mathrm{diag}(\theta_1,\ldots,\theta_m)$.
We sometimes identify each column ${}^t\!(t^j_1, \ldots, t^j_m)$ with the matrix $T_j$ itself.
Then we write
\[
\begin{array}{c}
x=u_i \\
\overbrace{\begin{array}{ccccc}
   T_0 & T_1 & \ldots & T_{b-1}  & \Theta \\
           \end{array}}
\end{array}.
\]
In this paper we call $t^i_j$'s and $\theta_j$'s \textit{exponents},
and we also call $\theta_j$'s \textit{characteristic exponents} (this terminology might be different from the usual one).
The table of the HTL forms represented by the above formula at all singular points is called the \textit{Riemann scheme} of a linear system.
In this paper we consider HTL forms whose residue matrices are diagonalizable.
We also assume that 
any two eigenvalues of $\Theta$ do not differ by a non-zero integer.

When we are not interested in the values of the exponents, a Riemann scheme is represented by a \textit{spectral type}.
The spectral type of an unramified linear system is described by the ``refining sequence of partitions''.
In the following we look at Riemann schemes and spectral types for particular cases.
For a general description of spectral types, see~\cite{K1, KNS}.

When $r_\nu=0$ for all singular points in (\ref{eq:rational_LDE}),
the Riemann scheme is merely a table of the eigenvalues of all $A_\nu^{(0)}$'s.
The spectral type is a tuple of partitions of $m$ which represents the multiplicities of the eigenvalues.
\begin{eg}
The Riemann scheme of a linear system
\[
\frac{dY}{dx}=\left( \frac{A_1^{(0)}}{x-u_1}+\frac{A_2^{(0)}}{x-u_2}+\frac{A_3^{(0)}}{x-u_3} \right)Y
\]
with the condition
\begin{equation*}
A_1^{(0)} \sim \mathrm{diag}(\theta^1_1, \theta^1_1, \theta^1_1, \theta^1_2), \quad
A_2^{(0)} \sim \mathrm{diag}(\theta^2_1, \theta^2_1, \theta^2_2, \theta^2_2), \quad
A_3^{(0)} \sim \mathrm{diag}(\theta^3_1, \theta^3_1, \theta^3_2, \theta^3_2),
\end{equation*}
and $A_\infty^{(0)}=-(A_1^{(0)}+A_2^{(0)}+A_3^{(0)}) \sim \mathrm{diag}(\theta^\infty_1, \theta^\infty_2, \theta^\infty_3, \theta^\infty_4)$ is
\[
\left(
\begin{array}{cccc}
  x=u_1 & x=u_2 & x=u_3 & x=\infty \\
   \begin{array}{c} \theta^1_1 \\ \theta^1_1 \\ \theta^1_1 \\ \theta^1_2 \end{array}
& \begin{array}{c} \theta^2_1 \\ \theta^2_1 \\ \theta^2_2 \\ \theta^2_2 \end{array}
& \begin{array}{c} \theta^3_1 \\ \theta^3_1 \\ \theta^3_2 \\ \theta^3_2 \end{array}
& \begin{array}{c} \theta^\infty_1 \\ \theta^\infty_2 \\ \theta^\infty_3 \\ \theta^\infty_4 \end{array}
\end{array}
\right).
\]
The spectral type of the system is $31,22,22,1111$.
\qed
\end{eg}
The next case is that one of the $r_\nu$'s is equal to one and the others are zero.
For example, we consider the case when $r_\infty=1$ and $r_\nu=0 \ (\nu=1, \ldots, n)$
\begin{equation}\label{eq:example_RS}
\frac{dY}{dx}=
\left(\sum_{\nu=1}^n \frac{A_\nu^{(0)}}{x-u_\nu}
+A_{\infty}^{(1)}
\right)Y,\quad
A_*^{(k)} \in M(m, \mathbb{C}).
\end{equation}
Its Riemann scheme can be obtained as follows.
Assume that $A_{\infty}^{(1)}$ is a diagonal matrix
\begin{align*}
A_{\infty}^{(1)}=
a_1 I_{m_1} \oplus \cdots \oplus a_k I_{m_k}, \quad
a_i \ne a_j \ (i \ne j), \quad m_1+\cdots+m_k=m,
\end{align*}
and the matrices $A_\nu^{(0)}$'s are similar to diagonal matrices $\Theta_\nu$'s respectively.
Partition the matrix $A_\infty^{(0)}=-\sum_{\nu=1}^n A_\nu^{(0)}$ into submatrices according to $A_{\infty}^{(1)}$
\begin{align*}
\begin{pmatrix}
\Theta^\infty_1 & A_{12} & \ldots & A_{1k} \\
A_{21} & \Theta^\infty_2 & \ldots & A_{2k} \\
\vdots & \vdots & \ddots & \vdots \\
A_{k1} & A_{k2} & \ldots & \Theta^\infty_k
\end{pmatrix}
\end{align*}
where $A_{ij}$ is an $m_i \times m_j$ matrix.
Here, using the conjugate action of
\[
\mathrm{Stab}( A_{\infty}^{(1)} )=
\left\{\, g \in \mathrm{GL}(m, \mathbb{C}) \,\mid\, g A_{\infty}^{(1)} g^{-1}=A_{\infty}^{(1)}\, \right\},
\]
we can choose $\Theta^\infty_j$'s so that they are diagonal.
Then the Riemann scheme of (\ref{eq:example_RS}) is given by
\[
\left(
\begin{array}{cccc}
  x=u_1 & \ldots & x=u_n & x=\infty \\
\Theta_1 & \ldots & \Theta_n
& \overbrace{\begin{array}{cc}
     -a_1 I_{m_1} & \Theta^\infty_1 \\
     \vdots & \vdots \\
     -a_k I_{m_k} & \Theta^\infty_k
        	\end{array}}
\end{array}
\right).
\]
Let $\lambda_\nu$ be the partition of $m$ which represents the multiplicities of the diagonal entries of $\Theta_\nu$.
Similarly, let $\mu_j$ be the partition of $m_j$ determined by the diagonal entries of $\Theta^\infty_j$.
Then the spectral type of~(\ref{eq:example_RS}) is
\[
\lambda_1, \, \ldots, \, \lambda_n, \, (\mu_1)(\mu_2)\cdots(\mu_k).
\]
Sometimes $(\mu_1)(\mu_2)\cdots(\mu_k)$ is written as a pair of two partitions $[\lambda, \mu]$.
Here $\lambda=m_1, \ldots, m_k$, 
and $\mu$ is the partition of $m$ obtained by arranging $\mu_1, \ldots, \mu_k$.

\subsection{Laplace transform}\label{sec:Laplace}
It is known that for any rational function matrix with $r_\infty=0$
\begin{equation}
A(x)=\sum_{\nu=1}^n\sum_{k=0}^{r_{\nu}}\frac{A_{\nu}^{(k)}}{(x-u_{\nu})^{k+1}}, \quad
A_{\nu}^{(k)} \in M(m, \mathbb{C}),
\end{equation}
there exist a natural number $n$,
an $n \times n$ matrix $T$, an $m \times n$ matrix $Q$, and
an $n \times m$ matrix $P$ such that the following holds
\begin{equation}
A(x)=Q(xI-T)^{-1}P.
\end{equation}
Moreover the quadruple $(n, T, Q, P)$ is essentially unique provided that $n$ is minimal (see \cite{Y1}).
In the following we often write a scalar matrix as $k$ instead of $kI$.
Thus we can write a system with $r_\infty=1$ in the form
\begin{equation}\label{eq:Laplace_before}
\frac{dY}{dx}=\left( Q(x-T)^{-1}P+S \right)Y.
\end{equation}

Using this expression, we can show that a system with $r_\infty=1$ behaves quite symmetric under the Laplace transform.
In fact, by applying the Laplace transform $x \mapsto -\frac{d}{dx}$, $\frac{d}{dx} \mapsto x$,
we obtain the following system~\cite{Y2}:
\begin{equation}\label{eq:Laplace_after}
\frac{dY}{dx}=\left( -P(x-S)^{-1}Q-T \right)Y.
\end{equation}
When the matrices $S$ and $T$ are both diagonalizable,
the correspondence of the Riemann schemes of (\ref{eq:Laplace_before}) and (\ref{eq:Laplace_after}) is as follows.

Let the rank of the system (\ref{eq:Laplace_before}) be $m$
and the size of $T$ be $n$.
We set 
\begin{align*}
S=s_1 I_{m_1} \oplus \cdots \oplus s_k I_{m_k}, \quad
T=t_1 I_{n_1} \oplus \cdots \oplus t_l I_{n_l}.
\end{align*}
We assume $m \le n$ (otherwise, consider the opposite correspondence of the following).
Under this assumption, $m_j \le n$ holds.
Moreover, with appropriate choice of $Q$ and $P$,
$n_j$ gives the rank of the residue matrix of the coefficient matrix of (\ref{eq:Laplace_before}) at $x=t_j$.
Thus we can assume $n_j \le m$ without loss of generality.
We partition $Q$ and $P$ as follows:
\begin{align*}
Q=
\begin{pmatrix}
Q_1 & \ldots & Q_l
\end{pmatrix}
=
\begin{pmatrix}
Q^1 \\
\vdots \\
Q^k
\end{pmatrix}, \quad
P=
\begin{pmatrix}
P_1 \\
\vdots \\
P_l
\end{pmatrix}
=
\begin{pmatrix}
P^1 & \ldots & P^k
\end{pmatrix}.
\end{align*}
Here $Q_j$ is an $m \times n_j$ matrix,
$P_j$ is an $n_j \times m$ matrix,
$Q^j$ is an $m_j \times n$ matrix, and
$P^j$ is an $n \times m_j$ matrix.
We assume that $Q_jP_j$ and $-Q^jP^j$ are similar to diagonal matrices $\Theta_j$ and $K_j$ respectively.
Then the Riemann scheme of (\ref{eq:Laplace_before}) is
\[
\left(
\begin{array}{cccc}
  x=t_1 & \ldots & x=t_l & x=\infty \\
\Theta_1 & \ldots & \Theta_l 
& \overbrace{\begin{array}{cc}
     -s_1 I_{m_1} & K_1 \\
     \vdots & \vdots \\
     -s_k I_{m_k} & K_k
        	\end{array}}
\end{array}
\right),
\]
and the Riemann scheme of (\ref{eq:Laplace_after}) is
\[
\left(
\begin{array}{cccc}
  x=s_1 & \ldots & x=s_k & x=\infty \\
\tilde{K}_1 & \ldots & \tilde{K}_k
& \overbrace{\begin{array}{cc}
     t_1 I_{n_1} & \tilde{\Theta}_1 \\
     \vdots & \vdots \\
     t_l I_{n_l} & \tilde{\Theta}_l
        	\end{array}}
\end{array}
\right).
\]
Here $\tilde{\Theta}_j$ denotes an $n_j \times n_j$ diagonal matrix obtained by eliminating $m-n_j$ zeros from $\Theta_j$,
and $\tilde{K}_j$ denotes an $n \times n$ diagonal matrix obtained by adding $n-m_j$ zeros to $K_j$.

\subsection{Schlesinger transformations}

The Schlesinger transformation \cite{Sc} was originally introduced as a discrete deformation of a Fuchsian system.
This discrete deformation corresponds to shifting the characteristic exponents by integers. 
For a detailed description of Schlesinger transformations, see \cite{JM}.
We will deal with Schlesinger transformations of a certain Fuchsian system in Section~\ref{sec:d-Garnier}.

A Schlesinger transformation of a Fuchsian system is realized as the compatibility condition of the Fuchsian system
and a system of linear difference equations: 
\begin{equation*}
\left\{
\begin{aligned}
\frac{dY}{dx}&=A(x)Y, \quad A(x)=\sum_{\nu=1}^n\frac{A_\nu}{x-u_\nu}, \\
\overline{Y}&=R(x)Y,
\end{aligned}
\right.
\end{equation*}
where the Fuchsian system is usually normalized so that the residue matrix at $x=\infty$ is diagonal.
In Section~\ref{sec:d-Garnier}, however, we will adopt a different gauge from the usual one.
Here $R=R(x)$ is a matrix whose entries are rational functions in $x$
and chosen so that the system of differential equations satisfied by $\overline{Y}$
\[
\frac{d\overline{Y}}{dx}=R[A(x)] \overline{Y}
\]
is a Fuchsian system with the same position of singular points as,
and the similar gauge to, the original system.
The matrix $R(x)$ is called the \textit{multiplier} of this transformation (or deformation).
Comparing $\overline{A}(x):=R[A(x)]$ with the original $A(x)$,
we have a system of difference equations satisfied by the entries of  $A_\nu$'s.

\section{Four-dimensional Painlev\'e-type difference equations}\label{sec:deg_scheme}
As mentioned in Section~\ref{sec:intro}, there are nine Fuchsian equations
which have four accessory parameters and three singular points.
They have only trivial continuous isomonodromic deformations
but admit non-trivial discrete isomonodromic deformations.
Therefore these nine Fuchsian equations can yield four-dimensional \textit{additive difference} Painlev\'e-type equations.

In this section we will see that the nine Fuchsian equations admit a ``degeneration scheme'' in some sense,
which is expected to give rise to a degeneration scheme of corresponding Painlev\'e-type difference equations.
To see this, we introduce an equivalence relation of spectral types.

Let $S_1$ and $S_2$ be spectral types. 
Then $S_1$ is said to be equivalent to $S_2$
if a linear system of spectral type $S_2$ is obtained by a finite number of applications of a M\"{o}bius transformation,
the Laplace transform, %(or its inverse),
and an addition 
from a linear system of spectral type $S_1$.
Here by ``addition'' we mean a shift of a coefficient matrix of (\ref{eq:rational_LDE}) by a scalar matrix
\[
A_{\nu}^{(k)} \mapsto A_{\nu}^{(k)}+\alpha I_m,
\]
which can be realized by a scalar gauge transformation.
We denote the equivalence class of a spectral type $S$ by $[S]$.
Moreover, we write $[S_1] \to [S_2]$ if there exist linear systems $E_1$ and $E_2$ whose spectral types are equivalent to 
$S_1$ and $S_2$ respectively such that $E_2$ is obtained from $E_1$ by means of a confluence of singular points or a degeneration of an HTL form.
For a detailed description of these two kinds of degenerations of linear systems, see~\cite{K1,K2,K3,KNS}.
\begin{eg}
We take the linear system of spectral type  $32,11111,11111$ as an example.
The confluence of singular points represented by $32$ and $11111$
leads to a non-Fuchsian system of spectral type $11111, (111)(11)$.
Applying a M\"{o}bius transformation if necessary, we can write the non-Fuchsian system in the form
\begin{equation}\label{eq:(111)(11),11111}
\frac{dY}{dx}=\left( \frac{QP}{x}+S \right)Y,
\end{equation}
which is a system (\ref{eq:Laplace_before}) with $T=O$.
Moreover, we can assume that the Riemann scheme of (\ref{eq:(111)(11),11111}) is of the form
\[
\left(
\begin{array}{cc}
  x=0 & x=\infty \\
\begin{array}{c} 0 \\ \alpha_1 \\ \alpha_2 \\ \alpha_3 \\ \alpha_4 \end{array}
& \overbrace{\begin{array}{cc}
     0 & \beta_1 \\
     0 & \beta_2 \\
     0 & \beta_3 \\
    -1 & \beta_4 \\
    -1 & \beta_5 \\
        	\end{array}}
\end{array}
\right),
\]
namely, we can adjust an exponent of maximum multiplicity at each singular point (except characteristic exponents at $x=\infty$)
to zero using additions.
Then, applying the Laplace transform to the non-Fuchsian system (\ref{eq:(111)(11),11111}),
we have a Fuchsian system of spectral type $211,1111,1111$ whose Riemann scheme is given by
\[
\left(
\begin{array}{ccc}
  x=0 &x=1 & x=\infty \\
\begin{array}{c} 0 \\ \beta_1 \\ \beta_2 \\ \beta_3 \end{array}
& \begin{array}{c} 0 \\ 0 \\ \beta_4 \\ \beta_5 \end{array}
& \begin{array}{c} \alpha_1 \\ \alpha_2 \\ \alpha_3 \\ \alpha_4 \end{array}
\end{array}
\right).
\]
The above relation of linear systems is illustrated in the following diagram:
\[
\xymatrix{
32,11111,11111  \ar[r]^{} & (111)(11), 11111  \ar@{=}[d]^{\text{\rotatebox{90}{$\sim$}}}_{\text{equivalent}} & \\
%  & & \\
  & 211,1111,1111 &
}
\]
This shows $[32,11111,11111] \to [211,1111,1111]$.
\qed
\end{eg}

\begin{rem}
The classification of Fuchsian systems with two accessory parameters is due to Kostov~\cite{Ko}:
\[
\begin{tabular}{r|ccc}
\#sing. \!\!pt. = 4 & $11,11,11,11$ & & \\
\hline
3 & $111,111,111$ & $22,1111,1111$ & $33,222,111111$
\end{tabular}
\]
The equivalence classes of the spectral types with three singular points in the above table admit the following degeneration scheme
\[
[33,222,111111] \to [22,1111,1111] \to [111,111,111].
\]
The difference Painlev\'e equations of type $A_2^{(1)*}$, $A_1^{(1)*}$, and $A_0^{(1)**}$ can be derived from
Fuchsian systems of these three spectral types
\begin{center}
$111,111,111$, \quad $22,1111,1111$, \quad $33,222,111111$,
\end{center}
respectively~\cite{B, DST, DT, Sak1}.
\qed
\end{rem}

By direct calculation,
we have the following

\begin{thm}\label{thm:scheme_difference}
The equivalence classes of the Fuchsian systems with four accessory parameters
and three singular points admit the following degeneration scheme.

\vspace{3mm}
{\small
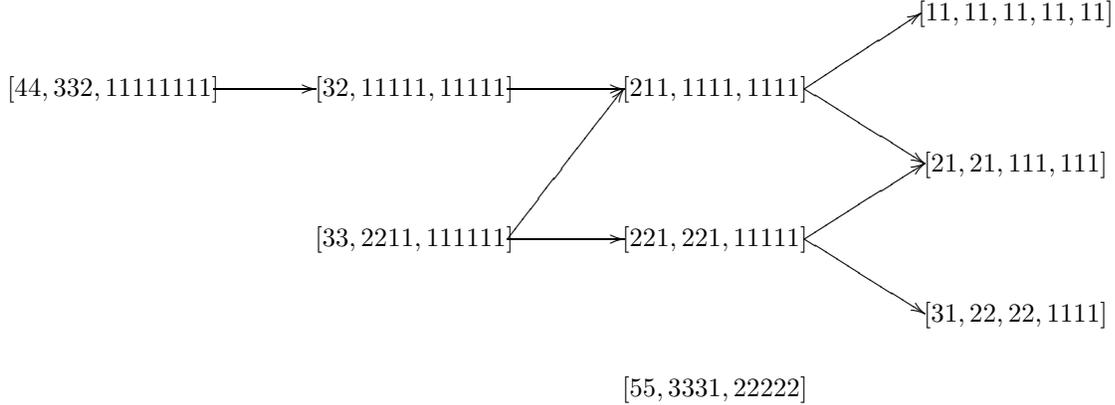
\begin{xy}
{(0,0) *{\begin{tabular}{c}
$[44,332,11111111]$
\end{tabular}
}},
{\ar (13.5,0);(27,0)},
{\ar (13.5,0);(27,-20)},
{\ar (12.5,-60);(28,-40)},
{\ar (12.5,-60);(28,-60)},
{(0,-60) *{\begin{tabular}{c}
$[66,444,2222211]$
\end{tabular}
}},
{(40,0) *{\begin{tabular}{c}
$[32,11111,11111]$
\end{tabular}
}},
{(40,-20) *{\begin{tabular}{c}
$[33,2211,111111]$
\end{tabular}
}},
{(40,-40) *{\begin{tabular}{c}
$[55,3331,22222]$
\end{tabular}}},
{(40,-60) *{\begin{tabular}{c}
$[44,2222,22211]$
\end{tabular}
}},
{\ar (52.5,0);(68,0)},
{\ar (52.5,-20);(68,-20)},
{\ar (52.5,-20);(68,0)},
{\ar (51.5,-60);(69,-60)},
{(80,0) *{\begin{tabular}{c}
$[211,1111,1111]$
\end{tabular}}},
{(80,-20) *{\begin{tabular}{c}
$[221,221,11111]$
\end{tabular}}},
{(80,-60) *{\begin{tabular}{c}
$[222,222,2211]$
\end{tabular}}},
{\ar (92,0);(107.5,10)},
{\ar (92,0);(108,-10)},
{\ar (92,-20);(108,-10)},
{\ar (92,-20);(108,-30)},
{\ar (91,-60);(109,-60)},
{(120,10) *{\begin{tabular}{c}
$[11,11,11,11,11]$
\end{tabular}}},
{(120,-10) *{\begin{tabular}{c}
$[21,21,111,111]$
\end{tabular}}},
{(120,-30) *{\begin{tabular}{c}
$[31,22,22,1111]$
\end{tabular}}},
{(120,-60) *{\begin{tabular}{c}
$[22,22,22,211]$
\end{tabular}}},
\end{xy}
}

%\vspace{5mm}
%
%{\small
%\begin{xy}
%{(0,0) *{\begin{tabular}{c}
%$[44,332,1^8]$
%\end{tabular}
%}},
%{\ar (10,0);(27.5,0)},
%{\ar (10,0);(25.5,-20)},
%{\ar (12,-50);(24,-50)},
%{\ar (12,-50);(25.5,-35)},
%{(0,-50) *{\begin{tabular}{c}
%$[66,444,2^511]$
%\end{tabular}
%}},
%{(36,0) *{\begin{tabular}{c}
%$[32,1^5,1^5]$
%\end{tabular}
%}},
%{(36,-20) *{\begin{tabular}{c}
%$[33,2211,1^6]$
%\end{tabular}
%}},
%{(36,-35) *{\begin{tabular}{c}
%$[55,3331,2^5]$
%\end{tabular}}},
%{(36,-50) *{\begin{tabular}{c}
%$[44,2^4,22211]$
%\end{tabular}
%}},
%{\ar (44,0);(62.5,0)},
%{\ar (46.5,-20);(61.5,-20)},
%{\ar (46.5,-20);(62.5,0)},
%{\ar (47.5,-50);(59,-50)},
%{(72,0) *{\begin{tabular}{c}
%$[211,1^4,1^4]$
%\end{tabular}}},
%{(72,-20) *{\begin{tabular}{c}
%$[221,221,1^5]$
%\end{tabular}}},
%{(72,-50) *{\begin{tabular}{c}
%$[222,222,2211]$
%\end{tabular}}},
%{\ar (81.5,0);(96.3,10)},
%{\ar (81.5,0);(96.9,-10)},
%{\ar (82.5,-20);(96.9,-10)},
%{\ar (82.5,-20);(96.9,-30)},
%{\ar (85,-50);(97.6,-50)},
%{(113,10) *{\begin{tabular}{c}
%$[11,11,11,11,11]$ $\cdots$
%\end{tabular}}},
%{(113,-10) *{\begin{tabular}{c}
%$[21,21,111,111]$ $\cdots$
%\end{tabular}}},
%{(113,-30) *{\begin{tabular}{c}
%$[31,22,22,1111]$ $\cdots$
%\end{tabular}}},
%{(113,-50) *{\begin{tabular}{c}
%$[22,22,22,211]$ $\cdots$
%\end{tabular}}},
%\end{xy}
%}
\end{thm}
%\begin{proof}
%Direct calculation.
%\end{proof}

\vspace{5mm}

\noindent
Notice that the Fuchsian system of spectral type $55, 3331, 22222$ does not admit confluences of singularities
since any one of the three partitions is not a refinement of the others (see \cite{KNS}).

The reason why we consider such an equivalence relation is
that continuous isomonodromic deformation equations are invariant under the Laplace transform.
A typical example is the Harnad duality~\cite{Har}.
We expect that an analogous statement also holds for discrete isomonodromic deformations.
\begin{conj}
Discrete isomonodromic deformation equations are invariant under the Laplace transform.
\end{conj}
\noindent
If the conjecture is true, then linear systems which are mutually equivalent give
the same Painlev\'e-type difference equation
since M\"obius transformations and additions clearly do not change
discrete isomonodromic deformation equations.
Therefore it is expected that Theorem~\ref{thm:scheme_difference} gives
the degeneration scheme of four-dimensional Painlev\'e-type difference equations.
Then the above scheme shows that the three streams starting from the Garnier system in two variables, the Fuji-Suzuki system, and
the Sasano system come from the same source.

To write down a Painlev\'e-type equation explicitly,
we have to parametrize linear systems which realize a given spectral type.
The following lemma might be useful for that purpose.
\begin{lem}\label{thm:separate}
Consider a linear system
\begin{align}\label{eq:system_given}
\frac{dY}{dx}=\left( \frac{A_0^{(1)}}{x^2}+\frac{A_0^{(0)}}{x}+\tilde{A}(x) \right)Y
\end{align}
where $\tilde{A}(x)$ is rational in $x$ and holomorphic at $x=0$.
Suppose that $A_0^{(1)}$ is diagonalizable and the spectral type at $x=0$ is $[\lambda, \mu]$.
Then there exists a linear system which has regular singular points at $x=0, \, -\varepsilon$
with spectral type $\lambda$, $\mu$ respectively and
tends to the system~(\ref{eq:system_given}) as $\varepsilon \to 0$.
\end{lem}
\begin{proof}
By a change of the dependent variable, we can assume that $A_0^{(1)}$ is diagonal:
\[
A_0^{(1)}=t_1 I_{m_1} \oplus \cdots \oplus t_k I_{m_k}.
\]
Further, using the action of $\mathrm{Stab}(A_0^{(1)})$, we can assume that $A_0^{(0)}$ has the form
\begin{equation}
A_0^{(0)}=
\begin{pmatrix}
\Theta_1 & A_{12} & \ldots & A_{1k} \\
A_{21} & \Theta_2 & \ldots & A_{2k} \\
\vdots & \vdots & \ddots & \vdots \\
A_{k1} & A_{k2} & \ldots & \Theta_k
\end{pmatrix}
\end{equation}
where $\Theta_j$ is an $m_j \times m_j$ diagonal matrix, $A_{ij}$ is an $m_i \times m_j$ matrix.

We define an upper triangular matrix $A_0$ and a lower triangular matrix $A_1$ as follows:
\begin{align*}
A_0=
\begin{pmatrix}
\rho_1 I_{m_1} & A_{12} &\ldots &A_{1k} \\
O & \rho_2 I_{m_2} & & A_{2k}\\
\vdots & & \ddots & \vdots \\
O & O & \ldots & \rho_k I_{m_k}
\end{pmatrix}, \quad
A_1=
\begin{pmatrix}
S_1 & O & \ldots & O \\
A_{21} & S_2 & & O \\
\vdots & & \ddots & \vdots \\
A_{k1} & A_{k2} & \ldots & S_k
\end{pmatrix}.
\end{align*}
Here $\rho_j=t_j/\varepsilon$ and $S_j=\Theta_j-(t_j/\varepsilon)I_{m_j}$.
Then we consider the following system
\begin{equation}\label{eq:separated}
\frac{dY}{dx}=
\left( \frac{A_0}{x}+\frac{A_1}{x+\varepsilon}+\tilde{A}(x) \right)Y.
\end{equation}
Noting that
\[
\frac{A_0}{x}+\frac{A_1}{x+\varepsilon}=\frac{\varepsilon A_0}{x(x+\varepsilon)}+\frac{A_0+A_1}{x+\varepsilon},
\]
we see that the system~(\ref{eq:separated}) is a desired one since it tends to (\ref{eq:system_given}) as $\varepsilon$ tends to zero.
\end{proof}
\begin{rem}
This upper-lower triangular gauge
appears naturally in the study of integrable systems.
For example, concerning isomonodromic deformations, see~\cite{H, MT, Sak2}.
\qed
\end{rem}

\begin{eg}
We can parametrize linear systems of spectral type $111,111,111$ using a parametrization of
$11,11,11,11$-systems.
Linear systems of spectral type $11,11,11,11$ are parametrized as follows~\cite{Sak2}
(here the choice of the gauge is slightly different from that in \cite{Sak2}, 
and we omit the gauge parameter):
\begin{align}\label{eq:11,11,11,11}
\frac{dY}{dx}=Q(x-T)^{-1}PY
\end{align}
where
\begin{align*}
T=\mathrm{diag}(t,1,0), \quad
Q=
\begin{pmatrix}
1 & 1 & 1 \\
tp & pq-\theta^\infty_2 & 0
\end{pmatrix}, \quad
P=
\begin{pmatrix}
\theta^t+pq & -q/t \\
\theta^1+\theta^\infty_2-pq & 1 \\
\theta^0 & \frac{q}{t}-1
\end{pmatrix}.
\end{align*}
Its Riemann scheme is
\[
\left(
\begin{array}{cccc}
  x=0 & x=1 & x=t  & x=\infty \\
\begin{array}{c} 0 \\ \theta^0 \end{array}
& \begin{array}{c} 0 \\ \theta^1  \end{array}
& \begin{array}{c} 0 \\ \theta^t  \end{array}
& \begin{array}{c} \theta^\infty_1 \\ \theta^\infty_2 \end{array}
\end{array}
\right).
\]
The Painlev\'e-type differential equation corresponding to (\ref{eq:11,11,11,11}) is the sixth Painlev\'e equation.
Applying the Laplace transform and a M\"obius transformation ($x \mapsto 1/x$)
to (\ref{eq:11,11,11,11}), we have
\begin{align}\label{eq:(1)(1)(1),111}
\frac{dY}{dx}=\left( \frac{T}{x^2}+\frac{PQ}{x} \right)Y.
\end{align}
It is easy to see that the spectral type of the system~(\ref{eq:(1)(1)(1),111}) is $(1)(1)(1), 111$.

By virtue of Lemma~\ref{thm:separate}, we can construct a Fuchsian system which gives
the system~(\ref{eq:(1)(1)(1),111}) through the confluence procedure:
\begin{align}\label{eq:111,111,111}
\frac{dY}{dx}=
\left( \frac{A_0}{x}+\frac{A_1}{x+\varepsilon} \right)Y
\end{align}
where
\begin{align*}
A_0=
\begin{pmatrix}
\rho_t & (PQ)_{12} & (PQ)_{13} \\
0 & \rho_1 & (PQ)_{23} \\
0 & 0 & 0
\end{pmatrix}, \quad
A_1=
\begin{pmatrix}
\sigma_t & 0 & 0 \\
(PQ)_{21} & \sigma_1 & 0 \\
(PQ)_{31} & (PQ)_{32} & \sigma_0
\end{pmatrix}.
\end{align*}
Then the Riemann scheme of the system~(\ref{eq:111,111,111}) is
\begin{equation}\label{eq:RS111,111,111}
\left(
\begin{array}{ccc}
  x=0 &x=-\varepsilon & x=\infty \\
\begin{array}{c} \rho_t \\ \rho_1 \\ 0 \end{array}
& \begin{array}{c} \sigma_t \\ \sigma_1 \\ \sigma_0  \end{array}
& \begin{array}{c} \theta^\infty_1 \\ \theta^\infty_2 \\ 0  \end{array}
\end{array}
\right),
\end{equation}
and the spectral type of which is $111,111,111$.
The relation of the exponents between (\ref{eq:(1)(1)(1),111}) and (\ref{eq:111,111,111}) is given by
\begin{align*}
\rho_t=\frac{t}{\varepsilon},
\ \ \rho_1=\frac{1}{\varepsilon},
\ \ \sigma_t=\theta^t-\frac{t}{\varepsilon},
\ \ \sigma_1=\theta^1-\frac{1}{\varepsilon},
\ \ \sigma_0=\theta^0.
\end{align*}
%Note that $p$ and $q$ are unchanged. 

Conversely, we can show by direct calculation that a linear system with the Riemann scheme~(\ref{eq:RS111,111,111})
can be written as (\ref{eq:111,111,111}).
Thus (\ref{eq:111,111,111}) gives a parametrization of linear systems of
spectral type $111,111,111$.

Similarly, using the parametrization of $111,111,111$-systems,
we can parametrize linear systems of spectral type $22,1111,1111$,
and using the parametrization of $22,1111,1111$,
we can obtain a parametrization of linear systems of spectral type $33,222,111111$.
\qed
\end{eg}

\section{Discrete analogue of the Garnier system}\label{sec:d-Garnier}
By considering Schlesinger transformations of Fuchsian equations which appear in Theorem~\ref{thm:oshima} (especially the ones with three singular points),
we can obtain four-dimensional Painlev\'e-type difference equations.
In this section, as an example, we calculate Schlesinger transformations of the Fuchsian system of spectral type $211,1111,1111$.
The confluence of two singular points of this Fuchsian system represented by $1111$ and $1111$ leads to the spectral type $(1)(1)(1)(1), 211$.
This is equivalent to $11,11,11,11,11$, which corresponds to the Garnier system in two variables.
On the other hand, the confluence of $211$ and $1111$ leads to $(11)(1)(1), 1111$.
This is equivalent to $21,21,111,111$, which corresponds to the Fuji-Suzuki system.
Here we consider the Schlesinger transformations corresponding to discretizations of
the two time evolutions of the Garnier system.

First we note the following proposition, 
which can be verified by direct calculation.
\begin{prop}\label{thm:LUgauge}
Let $A_0$ and $A_1$ be the following square matrices:
\begin{align*}
A_0=
\begin{pmatrix}
\lambda & {}^t\bm{0}_{m-1} \\
\bm{b} & B
\end{pmatrix}, \quad
A_1=
\begin{pmatrix}
\mu & {}^t\!\bm{c} \\
\bm{0}_{m-1} & C
\end{pmatrix}
\end{align*}
where $\lambda, \mu \in \mathbb{C}$, $\bm{b}, \bm{c} \in \mathbb{C}^{m-1}$, $B, C \in M(m-1, \mathbb{C})$,
with $B$ being upper triangular and $C$ lower triangular.
We have denoted by $\bm{0}_k$ the zero vector ${}^t\!(\overbrace{0, \ldots, 0}^{k})$.
Assume that $\lambda$ (resp. $\mu$) is not an eigenvalue of $B$ (resp. $C$).
Define the square matrix $G$ as
\begin{align*}
G&=
\begin{pmatrix}
l_1^{-1} & {}^t\bm{0} \\
\bm{0} & U_{22}
\end{pmatrix}
\begin{pmatrix}
1 & {}^t\!\bm{c}(\mu-C)^{-1} \\
-(\lambda-B)^{-1}\bm{b} & I_{m-1}
\end{pmatrix}.
\end{align*}
Here $l_1 \in \mathbb{C}^{\times}$, $U_{22} \in \mathrm{GL}(m-1, \mathbb{C})$ are determined by the
following LU decomposition
\begin{align}\label{eq:LUdec}
\begin{pmatrix}
1 & {}^t\!\bm{c}(\mu-C)^{-1} \\
\bm{0} & I_{m-1}
\end{pmatrix}
\begin{pmatrix}
1 & {}^t\bm{0} \\
(\lambda-B)^{-1}\bm{b} & I_{m-1}
\end{pmatrix}
=
\begin{pmatrix}
l_1 & {}^t\bm{0} \\
\bm{l} & L_{22}
\end{pmatrix}
\begin{pmatrix}
u_1 & {}^t\bm{u} \\
\bm{0} & U_{22}
\end{pmatrix},
\end{align}
where $U_{22}$ is upper triangular and $L_{22}$ lower triangular.
Then $GA_0G^{-1}$ is upper triangular and $GA_1G^{-1}$ is lower triangular.
More explicitly, we have
\begin{align*}
GA_0G^{-1}&=
\begin{pmatrix}
\lambda & -l_1^{-1}{}^t\!\bm{c}(\mu-C)^{-1}(\lambda-B)U_{22}^{-1} \\
\bm{0}_{m-1} & U_{22}BU_{22}^{-1}
\end{pmatrix}, \\
GA_1G^{-1}&=
\begin{pmatrix}
\mu & {}^t\bm{0}_{m-1} \\
-u_1^{-1}L_{22}^{-1}(\mu-C)(\lambda-B)^{-1}\bm{b} & L_{22}^{-1}CL_{22}
\end{pmatrix}.
\end{align*}
\end{prop}

\begin{rem}
We do not assume a particular normalization of the LU decomposition~(\ref{eq:LUdec}),
thus $l_1, L_{22}$, etc., are not unique
(see Remark~\ref{rem:uniqueness}).
\qed
\end{rem}

We parametrize linear systems of spectral type $211,1111,1111$.
Here we utilize a parametrization of linear systems of spectral type $11,11,11,11,11$~\cite{Sak2},
\begin{equation}\label{eq:11,11,11,11,11}
\frac{dY}{dx}=Q(x-T)^{-1}PY
\end{equation}
where
\begin{align*}
&T=\mathrm{diag}(t_1, t_2, 1, 0), \quad
U=\mathrm{diag}(1,u), \quad
W=\mathrm{diag}(w_1, w_2, w_3, w_4), \\
&Q=U^{-1}\hat{Q}W, \quad
\hat{Q}=
\begin{pmatrix}
1 & 1 & 1 & 1 \\
t_1p_1 & t_2p_2 & p_1q_1+p_2q_2-\theta^\infty_2 & 0
\end{pmatrix}, \\
&P=W^{-1}\hat{P}U, \quad
\hat{P}=
\begin{pmatrix}
\theta^{t_1}+p_1q_1 & -q_1/t_1 \\
\theta^{t_2}+p_2q_2 & -q_2/t_2 \\
\theta^1+\theta^\infty_2-p_1q_1-p_2q_2 & 1 \\
\theta^0 & \frac{q_1}{t_1}+\frac{q_2}{t_2}-1
\end{pmatrix}.
\end{align*}
Its Riemann scheme is
\[
\left(
\begin{array}{ccccc}
  x=0 & x=1 & x=t_1 & x=t_2 & x=\infty \\
\begin{array}{c} 0 \\ \theta^0 \end{array}
& \begin{array}{c} 0 \\ \theta^1  \end{array}
& \begin{array}{c} 0 \\ \theta^{t_1}  \end{array}
& \begin{array}{c} 0 \\ \theta^{t_2}  \end{array}
& \begin{array}{c} \theta^\infty_1 \\ \theta^\infty_2 \end{array}
\end{array}
\right).
\]
Applying the Laplace transform and a M\"obius transformation ($x \mapsto 1/x$) to (\ref{eq:11,11,11,11,11}),
we obtain
\begin{equation}\label{eq:(1)(1)(1)(1),211}
\frac{dY}{dx}=\left( \frac{T}{x^2}+\frac{PQ}{x} \right)Y.
\end{equation}
Then the Riemann scheme of (\ref{eq:(1)(1)(1)(1),211}) is
\[
\left(
\begin{array}{cc}
  x=0 & x=\infty \\
\overbrace{\begin{array}{cc}
     t_1 & \theta^{t_1} \\
     t_2 & \theta^{t_2} \\
    1    & \theta^1 \\
    0    & \theta^0 \\
        	\end{array}} & \begin{array}{c} \theta^\infty_1 \\ \theta^\infty_2 \\ 0 \\ 0 \end{array}
\end{array}
\right).
\]
The system (\ref{eq:11,11,11,11,11}) (and (\ref{eq:(1)(1)(1)(1),211})) has two deformation parameters $t_1, t_2$.
The (continuous) isomonodromic deformation of these systems is governed by the Garnier system in two variables:
\[
\frac{\partial q_i}{\partial t_j}=\frac{\partial H_j}{\partial p_i}, \quad
\frac{\partial p_i}{\partial t_j}=-\frac{\partial H_j}{\partial q_i} \quad (i,j=1,2)
\]
where the Hamiltonians are given by
\begin{align*}
 &t_i(t_i-1)H_i
\left({\theta^0, \theta^1  \atop \theta^{t_1}, \theta^{t_2}, \theta^\infty_2}
;{t_1\atop t_2};{q_1,p_1 \atop q_2,p_2}\right)\\
&=t_i(t_i-1)H_{\rm VI}\left({\theta^\infty_2, \theta^1 \atop
\theta^{t_i}, \theta^0+\theta^{t_{i+1}}+1} ; t_i ; q_i,p_i\right)
+(2q_ip_i+q_{i+1}p_{i+1}-\theta^1-2\theta^\infty_2)q_1q_2p_{i+1}\nonumber\\
&\quad-\frac{1}{t_i-t_{i+1}}
\{ t_i(t_i-1)(p_iq_i+\theta^{t_i})p_iq_{i+1}-t_i(t_{i+1}-1)(2p_iq_i+\theta^{t_i})p_{i+1}q_{i+1}\nonumber\\
&\quad+t_{i+1}(t_i-1)({p_{i+1}}^2q_{i+1}+\theta^{t_{i+1}}(p_{i+1}-p_i))q_i\}
\qquad (i\in \mathbb{Z}/2\mathbb{Z}).\nonumber
\end{align*}
Here $H_\mathrm{VI}$ stands for the Hamiltonian for the sixth Painlev\'e equation (see~\cite{KNS}).

By virtue of Lemma~\ref{thm:separate}, there exists a Fuchsian system
which gives the system~(\ref{eq:(1)(1)(1)(1),211}) through the confluence procedure:
\begin{align}\label{eq:211,1111,1111}
\frac{dY}{dx}=\left( \frac{A_0}{x}+\frac{A_1}{x+\varepsilon} \right)Y
\end{align}
where
\begin{align*}
A_0=
\begin{pmatrix}
\rho_{t_1} & (PQ)_{12} & (PQ)_{13} & (PQ)_{14} \\
0 & \rho_{t_2} & (PQ)_{23} & (PQ)_{24} \\
0 & 0 & \rho_1 & (PQ)_{34} \\
0 & 0 & 0 & 0
\end{pmatrix}, \quad
A_1=
\begin{pmatrix}
\sigma_1 & 0 & 0 & 0 \\
(PQ)_{21} & \sigma_2 & 0 & 0 \\
(PQ)_{31} & (PQ)_{32} & \sigma_3 & 0 \\
(PQ)_{41} & (PQ)_{42} & (PQ)_{43} & \sigma_4
\end{pmatrix}.
\end{align*}
Its Riemann scheme is
\[
\left(
\begin{array}{ccc}
  x=0 &x=-\varepsilon & x=\infty \\
\begin{array}{c} \rho_{t_1} \\ \rho_{t_2} \\ \rho_1 \\ 0 \end{array}
& \begin{array}{c} \sigma_1 \\ \sigma_2 \\ \sigma_3 \\ \sigma_4 \end{array}
& \begin{array}{c} \theta^\infty_1 \\ \theta^\infty_2 \\ 0 \\ 0 \end{array}
\end{array}
\right).
\]
This gives a parametrization of systems of spectral type $211,1111,1111$.
The relation of the exponents between (\ref{eq:(1)(1)(1)(1),211}) and (\ref{eq:211,1111,1111}) is given by
\begin{align*}
\rho_{t_j}=\frac{t_j}{\varepsilon} \ (j=1,2), \ \ 
\rho_1=\frac{1}{\varepsilon}, \ \ 
\sigma_j=\theta^{t_j}-\frac{t_j}{\varepsilon} \ (j=1,2), \ \ 
\sigma_3=\theta^1-\frac{1}{\varepsilon}, \ \ 
\sigma_4=\theta^0.
\end{align*}
Now we consider the Schlesinger transformations $S_1$ and $S_2$
\begin{align*}
S_1: \rho_{t_1} \mapsto \rho_{t_1}+1, \ \sigma_1 \mapsto \sigma_1-1, \quad
S_2: \rho_{t_2} \mapsto \rho_{t_2}+1, \ \sigma_2 \mapsto \sigma_2-1.
\end{align*}
Thus $S_j$ corresponds to the shift $t_j \mapsto t_j + \varepsilon$.

The multiplier $R_1$ for $S_1$ is given as follows:
\begin{align*}
R_1&=G_1\left( I_4+\frac{-\varepsilon E_1}{x+\varepsilon} \right), \quad
G_1=\hat{G}_1W, \quad E_1=\mathrm{diag}(1,0,0,0), \\
\hat{G}_1&=
\begin{pmatrix}
{l_1}^{-1} & {}^t \bm{0}_3 \\
\bm{0}_3 & U_1
\end{pmatrix}
\begin{pmatrix}
1 & {}^t\!\hat{\bm{c}}(\sigma_1-1-\hat{C})^{-1} \\
-(\rho_{t_1}+1-\hat{B})^{-1}\hat{\bm{b}} & I_3
\end{pmatrix},
\end{align*}
where
\begin{align*}
&{}^t\hat{\bm{b}}=\left( (\hat{P}\hat{Q})_{21}, (\hat{P}\hat{Q})_{31}, (\hat{P}\hat{Q})_{41} \right), \quad
{}^t\!\hat{\bm{c}}=\left( (\hat{P}\hat{Q})_{12}, (\hat{P}\hat{Q})_{13}, (\hat{P}\hat{Q})_{14} \right), \\
&\hat{B}=
\begin{pmatrix}
\rho_{t_2} & (\hat{P}\hat{Q})_{23} & (\hat{P}\hat{Q})_{24} \\
0 & \rho_1 & (\hat{P}\hat{Q})_{34} \\
0 & 0 & 0
\end{pmatrix}, \quad
\hat{C}=
\begin{pmatrix}
\sigma_2 & 0 & 0 \\
(\hat{P}\hat{Q})_{32} & \sigma_3 & 0 \\
(\hat{P}\hat{Q})_{42} & (\hat{P}\hat{Q})_{43} & \sigma_4
\end{pmatrix},
\end{align*}
and $l_1 \in \mathbb{C}^{\times}$, $U_1 \in \mathrm{GL}(3, \mathbb{C})$ are determined by the LU decomposition %(\ref{eq:LUdec}).
\begin{align*}
\begin{pmatrix}
1 & {}^t\!\hat{\bm{c}}(\sigma_1-1-\hat{C})^{-1} \\
\bm{0} & I_3
\end{pmatrix}
\begin{pmatrix}
1 & {}^t\bm{0} \\
(\rho_{t_1}+1-\hat{B})^{-1}\hat{\bm{b}} & I_3
\end{pmatrix}
=
\begin{pmatrix}
l_1 & {}^t\bm{0} \\
\bm{l}_1 & L_1
\end{pmatrix}
\begin{pmatrix}
u_1 & {}^t\bm{u}_1 \\
\bm{0} & U_1
\end{pmatrix}.
\end{align*}
We define the matrices $\overline{A}_0$, $\overline{A}_1$ by
\begin{align*}
\frac{\overline{A}_0}{x}+\frac{\overline{A}_1}{x+\varepsilon}&:=
R_1 \left[ \frac{A_0}{x}+\frac{A_1}{x+\varepsilon} \right]
=G_1\left( \frac{\tilde{A}_0}{x}+\frac{\tilde{A}_1}{x+\varepsilon} \right)G_1^{-1},
\end{align*}
where
\begin{align*}
\tilde{A}_0=
W^{-1}
\begin{pmatrix}
\rho_{t_1}+1 & {}^t\bm{0}_3 \\
\hat{\bm{b}} & \hat{B}
\end{pmatrix}
W, \quad
\tilde{A}_1=
W^{-1}
\begin{pmatrix}
\sigma_1-1 & {}^t\hat{\bm{c}} \\
\bm{0}_3 & \hat{C}
\end{pmatrix}
W.
\end{align*}
Notice that the desired shift of characteristic exponents is achieved.
Obviously we have $\tilde{A}_0+\tilde{A}_1=A_0+A_1$.
Proposition~\ref{thm:LUgauge} guarantees that $\overline{A}_0$ and $\overline{A}_1$ are again upper triangular and lower triangular, respectively.
%Thus $\overline{A}_*=S_1(A_*) \ (*=0,1)$.
Thus, by comparing $\overline{A}_0$, $\overline{A}_1$ with $A_0$, $A_1$,
we can obtain the discrete time evolution of $q_i, p_i, w_i$'s (which we denote by $\overline{q}_i, \overline{p}_i, \overline{w}_i$) along the $S_1$-direction.
To see this, we consider
the time evolution of $PQ$:
\begin{align*}
\overline{P} \, \overline{Q}=\overline{A}_0+\overline{A}_1=G_1(\tilde{A}_0+\tilde{A}_1)G_1^{-1}
=G_1(A_0+A_1)G_1^{-1}=\hat{G}_1\hat{P}\hat{Q}\hat{G}_1^{-1}.
\end{align*}
We set $M_1=\left( m^{(1)}_{ij} \right):=\hat{G}_1\hat{P}\hat{Q}\hat{G}_1^{-1}$.
Thus the difference equations satisfied by $q_i, p_i, w_i$ are determined by the equation
\begin{equation}\label{eq:d_Garnier}
\overline{P} \, \overline{Q}=M_1.
\end{equation}
Multiplying (\ref{eq:d_Garnier}) from the left by ${}^t\overline{\bm{w}}=(\overline{w}_1,  \overline{w}_2,  \overline{w}_3,  \overline{w}_4)$,
we see that $\overline{\bm{w}}$ satisfy
\begin{align*}
(\theta^\infty_1+{}^t\!M_1)\overline{\bm{w}}=\bm{0}.
\end{align*}
Since $\mathrm{rank}(\theta^\infty_1+{}^t\!M_1)=3$, $\overline{\bm{w}}$ is determined up to a scalar multiple.
From the $(1,4)$ and $(2,4)$ entries of the equation (\ref{eq:d_Garnier}) we then obtain
\begin{align}\label{eq:time_ev_pq}
\overline{p}_1\overline{q}_1=\frac{\overline{w}_1}{\overline{w}_4}m^{(1)}_{14}-\rho_{t_1}-\sigma_1, \quad
\overline{p}_2\overline{q}_2=\frac{\overline{w}_2}{\overline{w}_4}m^{(1)}_{24}-\rho_{t_2}-\sigma_2.
\end{align}
From the equation (\ref{eq:time_ev_pq}) and $(3,1)$, $(3,2)$ entries of (\ref{eq:d_Garnier}) we obtain the discrete time evolutions of $p_1$ and $p_2$:
\begin{align*}
&\frac{\rho_{t_1}+1}{\rho_1}\overline{p}_1
=\frac{\overline{w}_3}{\overline{w}_1}m^{(1)}_{31}+\frac{\overline{w}_1}{\overline{w}_4}m^{(1)}_{14}+\frac{\overline{w}_2}{\overline{w}_4}m^{(1)}_{24}
+\sigma_4+\theta^\infty_1, \\
&\frac{\rho_{t_2}}{\rho_1}\overline{p}_2
=\frac{\overline{w}_3}{\overline{w}_2}m^{(1)}_{32}+\frac{\overline{w}_1}{\overline{w}_4}m^{(1)}_{14}+\frac{\overline{w}_2}{\overline{w}_4}m^{(1)}_{24}
+\sigma_4+\theta^\infty_1.
\end{align*}
Thus the time evolutions of $q_1$ and $q_2$ are determined by
\begin{align*}
\overline{q}_i=\frac{\overline{p}_i\overline{q}_i}{\overline{p}_i} \quad (i=1,2),
\end{align*}
that is,
\begin{align*}
\frac{\rho_1}{\rho_{t_1}+1}
\overline{q}_1&=
\frac{\frac{\overline{w}_1}{\overline{w}_4}m^{(1)}_{14}-\rho_{t_1}-\sigma_1}
{\frac{\overline{w}_3}{\overline{w}_1}m^{(1)}_{31}+\frac{\overline{w}_1}{\overline{w}_4}m^{(1)}_{14}+\frac{\overline{w}_2}{\overline{w}_4}m^{(1)}_{24}
+\sigma_4+\theta^\infty_1}, \\
\frac{\rho_1}{\rho_{t_2}}
\overline{q}_2&=
\frac{\frac{\overline{w}_2}{\overline{w}_4}m^{(1)}_{24}-\rho_{t_2}-\sigma_2}
{\frac{\overline{w}_3}{\overline{w}_2}m^{(1)}_{32}+\frac{\overline{w}_1}{\overline{w}_4}m^{(1)}_{14}+\frac{\overline{w}_2}{\overline{w}_4}m^{(1)}_{24}
+\sigma_4+\theta^\infty_1}.
\end{align*}

\begin{rem}\label{rem:uniqueness} 
The matrix $M_1$ is determined up to conjugation by a diagonal matrix $D$
(see Proposition~\ref{thm:gauge_freedom} below).
Difference equations satisfied by $w_i$'s depend on this $D$.
However, the equations satisfied by $q_i$ and $p_i$ are determined independent of $D$,
since $\overline{w}_i$'s appear in the expressions of $\overline{q}_i$, $\overline{p}_i$
in the form of $\overline{w}_i m^{(1)}_{ij}/\overline{w}_j$.
Therefore $\overline{q}_i$, $\overline{p}_i$ are well-defined.
\qed
\end{rem}

\begin{prop}\label{thm:gauge_freedom}
Let $A_0$ (resp. $A_1$) $\in M(m, \mathbb{C})$ be an upper (resp. lower) triangular matrix with $m$ distinct eigenvalues.
Suppose that $g \in \mathrm{GL}(m, \mathbb{C})$ satisfy the following properties:
1) $gA_0g^{-1}$ is upper triangular,
2) $gA_1g^{-1}$ is lower triangular,
3) the order of the diagonal entries of $gA_*g^{-1}$ is the same as $A_* \ (*=0,1)$.
Then $g$ is diagonal.
\end{prop}

Next, we consider the multiplier $R_2$ for $S_2$.
Let $R_{20}$, $\tilde{A}_0$, and $\tilde{A}_1$ be the following matrices:
\begin{align*}
&R_{20}=
\left(
I_4+\frac{-\varepsilon E_2}{x+\varepsilon}
\right)
\left\{
\begin{pmatrix}
\rho_{t_1}-\rho_{t_2} & (PQ)_{12} \\
-(PQ)_{21} & \sigma_1-\sigma_2
\end{pmatrix}
\oplus
I_2
\right\}, \quad
E_2=\mathrm{diag}(0,1,0,0), \\
&\frac{\tilde{A}_0}{x}+\frac{\tilde{A}_1}{x+\varepsilon}:=
R_{20}\left[ \frac{A_0}{x}+\frac{A_1}{x+\varepsilon} \right].
\end{align*}
Define $\hat{\bm{b}}, \hat{\bm{c}} \in \mathbb{C}^2$ and $\hat{B}, \hat{C} \in M(2, \mathbb{C})$ by
\begin{align*}
\tilde{A}_0&=
%\begin{pmatrix}
%\rho_{t_1} & (PQ)_{12} & * & * \\
%0 & \rho_{t_2}+1 & 0 & 0 \\
%0 & * & \rho_1 & (PQ)_{34} \\
%0 & * & 0 & 0
%\end{pmatrix}
%=
%\begin{pmatrix}
%\rho_{t_1} & * & ** \\
%0 & \rho_{t_2}+1 & {}^t\bm{0}_2 \\
%\bm{0}_2 & \bm{b} & B
%\end{pmatrix}
%=
W^{-1}
\begin{pmatrix}
\rho_{t_1} & * & ** \\
0 & \rho_{t_2}+1 & {}^t\bm{0}_2 \\
\bm{0}_2 & \hat{\bm{b}} & \hat{B}
\end{pmatrix}W, \quad
\tilde{A}_1=
%\begin{pmatrix}
%\sigma_1 & 0 & 0 & 0 \\
%(PQ)_{21} & \sigma_2-1 & * & * \\
%* & 0 & \sigma_3 & 0 \\
%* & 0 & (PQ)_{43} & \sigma_4
%\end{pmatrix}
%=
%\begin{pmatrix}
%\sigma_1 & 0 & {}^t\bm{0}_2 \\
%* & \sigma_2-1 & {}^t\!\bm{c} \\
%\ \rotatebox{90}{**} & \bm{0}_2 & C
%\end{pmatrix}
%=
W^{-1}
\begin{pmatrix}
\sigma_1 & 0 & {}^t\bm{0}_2 \\
* & \sigma_2-1 & {}^t\!\hat{\bm{c}} \\
\ \rotatebox{90}{**} & \bm{0}_2 & \hat{C}
\end{pmatrix}
W.
\end{align*}
The multiplier $R_2$ is given by
\begin{align*}
R_2&=G_2R_{20}, \quad
G_2=\left( (1) \oplus \hat{G}_2 \right)W, \\
\hat{G}_2&=
\begin{pmatrix}
{l_2}^{-1} & {}^t\bm{0}_2 \\
\bm{0}_2 & U_2
\end{pmatrix}
\begin{pmatrix}
1 & {}^t\!\hat{\bm{c}}(\sigma_2-1-\hat{C})^{-1} \\
-(\rho_{t_2}+1-\hat{B})^{-1}\hat{\bm{b}} & I_2
\end{pmatrix},
\end{align*}
where $l_2 \in \mathbb{C}^{\times}$ and $U_2 \in \mathrm{GL}(2, \mathbb{C})$ are determined by
\begin{align*}
\begin{pmatrix}
1 & {}^t\!\hat{\bm{c}}(\sigma_2-1-\hat{C})^{-1} \\
\bm{0} & I_2
\end{pmatrix}
\begin{pmatrix}
1 & {}^t\bm{0} \\
(\rho_{t_2}+1-\hat{B})^{-1}\hat{\bm{b}} & I_2
\end{pmatrix}
=
\begin{pmatrix}
l_2 & {}^t\bm{0} \\
\bm{l}_2 & L_2
\end{pmatrix}
\begin{pmatrix}
u_2 & {}^t\bm{u}_2 \\
\bm{0} & U_2
\end{pmatrix}.
\end{align*}
We define $\overline{A}_0$ and $\overline{A}_1$ by
\begin{align*}
\frac{\overline{A}_0}{x}+\frac{\overline{A}_1}{x+\varepsilon}&:=
R_2\left[ \frac{A_0}{x}+\frac{A_1}{x+\varepsilon} \right]
=G_2\left( \frac{\tilde{A}_0}{x}+\frac{\tilde{A}_1}{x+\varepsilon} \right)G_2^{-1}.
\end{align*}
Then $\overline{A}_0$ and $\overline{A}_1$ are upper and lower triangular, respectively. 
Note that we use the same symbol as the $S_1$-direction to denote the time evolution along the $S_2$-direction.
In this case
$\overline{P} \, \overline{Q}=\overline{A}_0+\overline{A}_1=G_2(\tilde{A}_0+\tilde{A}_1)G_2^{-1}$.
We set $M_2=\left( m^{(2)}_{ij} \right):=G_2(\tilde{A}_0+\tilde{A}_1)G_2^{-1}$.
Unlike the case of $S_1$, the matrix $\tilde{A}_0+\tilde{A}_1$ is not equal to $A_0+A_1$,
but $\mathrm{rank}(\theta^\infty_1+{}^t\!M_2)=3$ holds.
In the same way as in the case of $S_1$, we obtain the following  
\begin{align*}
\frac{\rho_{t_1}}{\rho_1}\overline{p}_1
&=\frac{\overline{w}_3}{\overline{w}_1}m^{(2)}_{31}+\frac{\overline{w}_1}{\overline{w}_4}m^{(2)}_{14}+\frac{\overline{w}_2}{\overline{w}_4}m^{(2)}_{24}
+\sigma_4+\theta^\infty_1, \\
\frac{\rho_{t_2}+1}{\rho_1}\overline{p}_2
&=\frac{\overline{w}_3}{\overline{w}_2}m^{(2)}_{32}+\frac{\overline{w}_1}{\overline{w}_4}m^{(2)}_{14}+\frac{\overline{w}_2}{\overline{w}_4}m^{(2)}_{24}
+\sigma_4+\theta^\infty_1, \\
\frac{\rho_1}{\rho_{t_1}}
\overline{q}_1&=
\frac{\frac{\overline{w}_1}{\overline{w}_4}m^{(2)}_{14}-\rho_{t_1}-\sigma_1}
{\frac{\overline{w}_3}{\overline{w}_1}m^{(2)}_{31}+\frac{\overline{w}_1}{\overline{w}_4}m^{(2)}_{14}+\frac{\overline{w}_2}{\overline{w}_4}m^{(2)}_{24}
+\sigma_4+\theta^\infty_1}, \\
\frac{\rho_1}{\rho_{t_2}+1}
\overline{q}_2&=
\frac{\frac{\overline{w}_2}{\overline{w}_4}m^{(2)}_{24}-\rho_{t_2}-\sigma_2}
{\frac{\overline{w}_3}{\overline{w}_2}m^{(2)}_{32}+\frac{\overline{w}_1}{\overline{w}_4}m^{(2)}_{14}+\frac{\overline{w}_2}{\overline{w}_4}m^{(2)}_{24}
+\sigma_4+\theta^\infty_1}.
\end{align*}
Note that these time evolutions are rather complicated when written explicitly in terms of $q_i, p_i$'s.

There remains some ambiguity regarding time evolutions of $w_i$'s.
Although this ambiguity is resolved by the compatibility condition of the $S_1$-direction and the $S_2$-direction,
we do not go into the details.

\begin{rem}
A construction in this section can be generalized to the Garnier system in $N$-variables.
Discrete analogues of the Garnier systems were also considered in~\cite{OR},
where they were derived from $2 \times 2$ linear systems.
The relationship between the two constructions is not well understood at this time.
\qed
\end{rem}

\begin{rem}
The following Schlesinger transformation of (\ref{eq:211,1111,1111}):
\[
\theta^\infty_1 \mapsto \theta^\infty_1+1, \quad \sigma_1 \mapsto \sigma_1-1
\]
gives a discrete analogue of the Fuji-Suzuki system.
\qed
\end{rem}

\begin{rem}
For any Painlev\'e-type differential equation corresponding to a Fuchsian system, in the same way as in Section~\ref{sec:d-Garnier},
we can construct a system of difference equations which can be regarded as a discrete analogue of the Painlev\'e-type differential equation.
\qed
\end{rem}

\end{document}